\documentclass[11pt]{amsart}
\usepackage{graphicx,fullpage,float}
\usepackage{amsmath,amssymb,amsthm,commath,mathtools,mathrsfs,url,xcolor,hyperref,tabularx}
\usepackage{orcidlink,cleveref}
\usepackage{subfigure}

\usepackage[
    backend=biber, sortcites=true,
    style=numeric, giveninits=true, maxbibnames=99
  ]{biblatex}
\newtheorem{theorem}{Theorem}%
\newtheorem{proposition}[theorem]{Proposition}%
\newtheorem{lemma}[theorem]{Lemma}
\newtheorem{corollary}[theorem]{Corollary}
\newtheorem{example}{Example}%
\newtheorem{remark}{Remark}%
\newtheorem{definition}{Definition}%
\newtheorem*{conjecture}{Conjecture}%

\newcommand{\conv}{\mathrm{conv}}

\newcommand{\GL}{\operatorname{GL}}
\newcommand{\Diag}{\operatorname{Diag}}

\newcommand\ZZ{\mathbb{Z}}
\newcommand\RR{\mathbb{R}}
\newcommand\CC{\mathbb{C}}
\newcommand\QQ{\mathbb{Q}}
\newcommand\PP{\mathbb{P}}

\begin{document}
\title[Semigroups of Integer Points  in Convex Cones]{Semigroups of Integer Points in Convex Cones} 
\author{Grigoriy Blekherman, Jesús A. De Loera, Luze Xu, Shixuan Zhang}

\begin{abstract}
    We study the question whether the affine semigroup of integer points in a convex cone can be finitely generated up to symmetries of the cone. We establish general properties of finite generation up to symmetry, and then concentrate on the case of irrational polyhedral cones.
\end{abstract}
\maketitle              
\section{Introduction}

An affine semigroup $S$ is a subset of $\mathbb{Z}^n$ that contains $\mathbf{0}$ and is closed under addition (Some authors also call the same objects \emph{monoids} as they are submonoids of the lattice). Given a convex cone $C\subseteq \mathbb{R}^n$, the integer points $S_C:=C\cap\mathbb{Z}^n$ form a semigroup which we will call the \emph{conical semigroup} of $C$. An important example comes from any compact convex body $K\subseteq\RR^n$, the integer points of the $\text{cone}(K\times \{1\})\cap\ZZ^{n+1}$ form a conical semigroup. 

Conical semigroups appear in many areas of mathematics from the pure to the applied: 
In the theory of \emph{toric varieties} and \emph{combinatorial commutative algebra}; the coordinate rings of affine toric varieties are given by the algebras generated by the conical semigroups for rational cones.
\cite{BrunsGubeladzenotes2002,BrunsGubeladze2003,BrunsGubeladzeTrung2002,cox2005using,cox2024toric,fulton1993introduction}. More recently, generalizing the notion of Newton polytope, several authors \cite{LazarsfeldMustata2009,Kaveh+Khovanskii:semigroups}
developed the theory of \emph{Newton-Okounkov bodies}. A Newton-Okounkov body is a convex set associated to a variety equipped with auxiliary data, namely a valuation. In \cite{Kaveh+Khovanskii:semigroups} they showed that any semigroup in the lattice $\mathbb{Z}^n$ is in fact asymptotically approximated by the semigroup of the points in a sublattice and lying in a convex cone. As a consequence, for a large class of graded algebras, their Hilbert functions have polynomial growth and their growth coefficients satisfy a Brunn-Minkowski type inequality. Their theory can also be used to provide toric degenerations for general varieties and obtains a generalization of Kushnirenko's theorem.

Another topic where the lattice points of convex cones play an important role is {\em Minkowski's geometry of numbers} \cite{GeometrieZahlenGL}. Concretely, in the study of the space of positive definite quadratic forms (which forms a convex cone) and sphere packings on associated lattices. For example, the classical Diophantine problem for an integral quadratic form $Q$ is to decide for which integers $n$ the equation $Q(x)=n$ has an integral solution $x$. The minimum representable integer $n$, is the \emph{arithmetical minimum} of $Q$ and upper bounds for its value depend on properties of the lattice associated to $Q$ (see Chapters 1, 2 \cite{SchurmannBook}).
In applied mathematics, the cone of positive semidefinite matrices is of importance in convex optimization \cite{blekherman2012semidefinite} and the lattice points of convex cones are key to algorithms in integer convex optimization. In particular finite generation for rational polyhedral cones is crucial in many algorithms \cite{MR3835599,berndt2023new,MR4391780,MR3633776,deLoera2025integer,demeijer2023integrality}.

One of the most fundamental problems in semigroup theory is to investigate 
the generators of a semigroup. It was proved by Hilbert \cite{hilbert_ueber_1890} that the conical semigroups of rational polyhedral cones are finitely generated and the minimal generating sets are called \emph{Hilbert bases} \cite{deloera2012algebraic,schrijver1998theory}. This fact has been used in a variety of settings,  for example in the theory of toric varieties \cite{fulton1993introduction,cox2024toric} and in algebraic combinatorics \cite{borda2018lattice,stanley2007combinatorics}.

Our paper explores the structure of conical semigroups beyond the well-studied case of rational polyhedral cones. For irrational polyhedral cones we found an unexpected connection between the question of finite generation up to symmetry and a classical result in algebraic number theory: the \emph{Dirichlet unit theorem} (see \cite{algnumberthybook}). A related set of ideas, going under the name ``Noetherianity up to symmetry", was explored in the context of symmetric ideals in the polynomial ring with infinitely many variables \cite{MR2327026, MR2854168, MR3418745, MR3329086, MR3659339, MR3666212}.

Although conical semigroups of cones that are not rational polyhedral cannot be finitely generated in the usual sense, if a cone $C$ has enough integer matrix symmetries, then we can use a finitely generated subgroup $G$ of the group of integer symmetries to ensure finite generation.
Because the possibly infinite generators for $S_C$ can be obtained by group action $G$ on a finite set $R$ and $G$ is finitely generated, this also allows for the possibility of algorithmic methods. In \cite{deLoera2025integer} we formally introduced a new notion of finite generation for conical semigroups.

We denote by $\mathrm{GL}(n,\ZZ)$ the group of \emph{unimodular matrices} $=\{U\in\ZZ^{n\times n}:~|\det(U)|=1\}$. 

\begin{definition}\label{def:RG-FG}
Given a conical semigroup $S_C\subset\ZZ^n$, we call it $(R,G)$-finitely generated if there is a finite subset $R\subseteq S_C$ and a finitely generated subgroup $G\subseteq \mathrm{GL}(n, \ZZ)$ acting on $C$ linearly such that 
\begin{enumerate}
    \item both the cone $C$ and the semigroup $S_C$ are invariant under the group action, i.e., $G\cdot C=C$ and $G\cdot S_C =S_C$, and
    \item every element $s\in S_C$ can be represented as \[s =  \displaystyle\sum_{i\in K}\lambda_i T_{i}\cdot r_i\] for some $r_i \in R$, $T_{i}\in G$, and $\lambda_i \in \mathbb{Z}_{\geq 0}$, and where $K$ is a finite index set.
\end{enumerate}
\end{definition}

Note that when $C$ is a (pointed) rational polyhedral cone, then the conical semigroup $S_C=C\cap\ZZ^n$ is $(R, G)$-finitely generated by $R$, its \emph{Hilbert basis}, and $G$, the trivial group $\{I_n\}$. 
If $S_C$ is $(R,G)$-finitely generated, we also call the cone $C$ $(R,G)$-finitely generated.
In \cite{deLoera2025integer} the authors showed that some \emph{Lorentz cones} and all the \emph{cones of positive semidefinite matrices} are $(R,G)$-finitely generated.

In the rest of the paper, we use $[u]$ to denote the ray generated by a vector $u\in\RR^n$, i.e., 
$[u]=[u']$ whenever there exists $c>0$ such that $u=c\cdot u'$. We say an extreme ray $[u]$ is \emph{rational} if $[u] = [u']$ for some rational vector $u'\in\QQ^n$ and irrational otherwise. Polyhedral cones that have at least one irrational extreme ray are called irrational. Irrational 
cones and polyhedra are of interest in algebraic combinatorics (Ehrhart theory) \cite{borda2018lattice,cristofaro2019irrational} and in the theory of toric varieties \cite{postinghel2015degenerations,pir2022irrational}.

\subsection{Our Contributions}
In~\Cref{sec:properties} we present a set of key properties of $(R,G)$-finitely generated cones. Among other results we prove that for a nonpolyhedral cone to have an $(R,G)$-finitely generated conical semigroup, either it has infinitely many rational extreme rays, or all of its rational extreme rays lie in a proper subspace.

\begin{theorem}\label{thm:FiniteRationalRays}
    Suppose $C\subset\RR^n$ is a full-dimensional nonpolyhedral convex cone such that there are only finitely many extreme rays $[u_1],\dots,[u_m]$ of $C$ with $u_1,\dots,u_m\in\QQ^n$ and $\operatorname{span}_\RR\{u_1,\dots,u_m\}=\RR^n$.
    Then $C\cap\ZZ^n$ is not $(R,G)$-finitely generated.
\end{theorem}

The argument is presented in~\Cref{sec:nonpolyhedral}. 
A rather beautiful consequence of our method is that the Fermat cones is not $(R,G)$-finitely generated:
 
\begin{example}[Fermat cones]\label{example:Fermat}
    The Fermat cone $F_k:=\{(x,y,z)\in\RR^3: x^{2k} + y^{2k}\leq z^{2k}, z\geq 0\}$ for any $k\ge 2$ is not $(R,G)$-finitely generated. 
    This follows from~\Cref{thm:FiniteRationalRays}, since by Fermat's last theorem~\cite{wiles1995modular},  the only integral extreme rays of $F_k$ are  $\{(\pm1,0,1),(0,\pm1,1)\}$.
\end{example}

We can significantly generalize this observation via a different method. 
For a closed convex full-diemnsional cone $C\subset\RR^n$ defined by polynomial inequalities, whether $C$ can be $(R,G)$-finitely generated also depends on its \emph{algebraic boundary}, i.e., the Zariski closure of its topological boundary in $\CC^n$.
Since $C$ is a cone, its algebraic boundary can be viewed as a projective hypersurface in $\PP^{n-1}:=\PP(\CC^n)$ defined by a single homogeneous polynomial.
We say the algebraic boundary of $C$ is \emph{smooth} if the defining homogeneous polynomial is nonsingular over $\CC$.
\begin{theorem}\label{thm:SmoothHypersurfaceCone}
    Suppose $C\subset\RR^n$ is a full-dimensional closed convex cone such that its algebraic boundary is a smooth projective hypersurface defined by a homogeneous polynomial of degree $d$.
    If $n\ge4$, $d\ge3$, and $(n,d)\neq(4,4)$, then $C\cap\ZZ^n$ is not $(R,G)$-finitely generated.
\end{theorem}

An immediate consequence is that the conical semigroups in most (weighted) $\ell^p$-norm cones are not $(R,G)$-finitely generated.

\begin{example}[weighted $\ell^p$-norm cones]\label{example:norm_cones}
    Consider any weighted $\ell^p$-norm cone $C:=\{(x_1,\dots,x_n)\in\RR^n:x_n^{p}\ge a_1x_1^{p}+\cdots+a_{n-1}x_{n-1}^{p},x_n\ge0\}$ for any even order $p$, and positive weights $a_1,\dots,a_{n-1}$.
    Then by~\Cref{thm:SmoothHypersurfaceCone}, $C\cap\ZZ^n$ is not $(R,G)$-finitely generated for any $n\ge4$, $p\ge4$ and $(n,p)\neq(4,4)$, since the the algebraic boundary of $C$ is defined by the smooth homogeneous polynomial $x_n^{p}-a_1x_1^p-\cdots-a_{n-1}x_{n-1}^p$.
\end{example}

We now discuss how ~\Cref{thm:SmoothHypersurfaceCone} fits with some previously known results. The cone of positive semidefinite matrices was shown to be finitely generated in~\cite{deLoera2025integer}. Its
algebraic boundary is given by the determinant, which is either singular or quadratic, and so the above Theorem does not apply.
The algebraic boundary of the Lorentz (second-order) cone also discussed in ~\cite{deLoera2025integer} is quadratic. We note that for polyhedral cones, the algebraic boundaries are reducible and hence singular for $n\geq 3$. We plan to address the question of $(R,G)$-finite generation for quadratic cones in upcoming work.

In the plane, all convex cones are polyhedral, and we provide a complete characterization of $(R,G)$-finitely generated cones.

\begin{theorem}\label{thm:mother2d}
    Suppose $C\in\RR^2$ is a convex cone.
    \begin{enumerate}
   \item If $C$ is pointed with extreme rays $[u_1], [u_2]$, where $u_1, u_2\in\RR^2$. Then $S_C$ is $(R,G)$-finitely generated, if and only if either $[u_1], [u_2]$ are both rational; or $u_1, u_2$ are two eigenvectors with positive eigenvalues of a non-identity unimodular matrix $A\in\mathrm{GL}(2,\ZZ)$.
   \item If $C$ is a half-space associated with the line spanned by $r\in\RR^2$. Then $S_C$ is $(R,G)$-finitely generated, if and only if $r$ is rational or $r$ is the eigenvector with positive eigenvalue of a non-identity unimodular matrix $A\in\mathrm{GL}(2, \ZZ)$ with $\det(A)=1$.
 \end{enumerate}
\end{theorem}

The proof of Theorem \ref{thm:mother2d} will be done using Theorems \ref{thm:main2d} and \ref{thm:halfspace}.

In view of~\Cref{thm:FiniteRationalRays,thm:SmoothHypersurfaceCone,thm:mother2d}, we tackle the next important case in terms of cone complexity: that of irrational polyhedral cones.
We show that irrational polyhedral cones, that are never finitely generated in the traditional Hilbert's sense, can be $(R,G)$-finitely generated, and we come close to a full characterization of $(R,G)$-finitely generated simple irrational polyhedral cones (i.e.,  irrational polyhedral cones over a simplex).

\begin{theorem}\label{thm:polyhedral-necessary}
    Let $C\subset\RR^n$ be a polyhedral cone with extreme rays $[u_1],\dots,[u_m]$.
    If $C$ is $(R,G)$-finitely generated, then either $[u_1],\dots,[u_m]$ are all rational, or there exists a non-identity matrix $A\in\GL(n,\ZZ)$ such that $Au_i=\lambda_iu_i$ for some $\lambda_i\in\RR_{>0}$, $i=1,\dots,m$.
\end{theorem}

Then we proceed to a sufficient condition for simple irrational polyhedral cones to be $(R,G)$-finitely generated.

\begin{theorem}\label{thm:simple-sufficient}
    Let $C\subset\RR^n$ be a simple polyhedral cone with extreme rays $[u_1],\dots,[u_n]$, where $[u_1],\dots,[u_k]$ are rational.
    Then $C$ is $(R,G)$-finitely generated if there exists $A\in\QQ^{n\times n}$ and $\lambda_1,\dots,\lambda_n\in\RR_{>0}$ such that $Au_i=\lambda_iu_i$ for $i=1,\dots,n$ and $\lambda_{k+1},\dots,\lambda_n$ are distinct.
\end{theorem}

The proofs of~\Cref{thm:polyhedral-necessary,thm:simple-sufficient} are provided in~\Cref{sec:main-polyhedral}. A key technical tool in the proof of \Cref{thm:simple-sufficient} is the Dirichlet unit theorem which allows to conclude existence of many unimodular integer matrices which have extreme rays as eigenvectors. The key difference between these two conditions is that in~\Cref{thm:simple-sufficient}, we require the cone to be simple and the eigenvalues of $A$ associated with irrational extreme rays to be distinct, while in~\Cref{thm:polyhedral-necessary} the cone $C$ is any polyhedral cone and the matrix $A$ can have repeated eigenvalues. We note that in the case of repeated eigenvalues the (generalized version of) Dirichlet unit theorem does not apply.
In~\Cref{sec:3dim}, we close the gap between the two conditions and show that a 3-dimensional irrational polyhedral cone is $(R,G)$-finitely generated if and only if it is simple and satisfies the sufficient condition of Theorem \ref{thm:simple-sufficient}.
In~\Cref{sec:4dim-example}, we show that some examples of 4-dimensional simple cones satisfy the necessary condition in Theorem~\ref{thm:polyhedral-necessary}, but are not $(R,G)$-finitely generated due to delicate considerations caused by repeated eigenvalues. 
In fact, we have the tantalizing conjecture that the sufficient condition of Theorem \ref{thm:simple-sufficient} is necessary.

\begin{conjecture}
Let $C\subset\RR^n$ be a simple polyhedral cone with extreme rays $[u_1],\dots,[u_n]$, where $[u_1],\dots,[u_k]$ are rational.
Then $C$ is $(R,G)$-finitely generated if and only if there exists $A\in\QQ^{n\times n}$ and $\lambda_1,\dots,\lambda_n\in\RR_{>0}$ such that $Au_i=\lambda_iu_i$ for $i=1,\dots,n$ and $\lambda_{k+1},\dots,\lambda_n$ are distinct.
\end{conjecture}

\section{General cones and \texorpdfstring{$(R,G)$}{(R,G)}-finite generation}

\subsection{Properties of \texorpdfstring{$(R,G)$}{(R,G)}-finite generation}\label{sec:properties}

We begin with some general properties about $(R,G)$-finitely generated cones.
While in~\Cref{def:RG-FG} we allow the unimodular matrix to act arbitrarily on the cone, it turns out for full dimensional cones such action can always be represented by multiplication of a unimodular matrix as shown below.
\begin{lemma}
    If $S_C$ is $(R,G)$-finitely generated, then for any $T\in G$, $T\cdot C = C$.
\end{lemma}
\begin{proof}
    From Definition~\ref{def:RG-FG} (first condition), we must have $T(C)\subseteq C$.
    Since $G$ is a group,  $T^{-1}\in G$ so similarly $T^{-1}(C)\subseteq C$, which implies that $C=T(T^{-1}(C))\subseteq T(C)$.
    Combining these two containment relations shows the desired equality.
\end{proof}

\begin{lemma}\label{lem:unimodular}
    If $C\subset\RR^n$ is a full-dimensional cone and $S_C$ is $(R,G)$-finitely generated, then for any $T\in G$, there exists a unimodular matrix $A_T\in\ZZ^{n\times n}$ such that $T\cdot x = A_T x$ for any $x\in C$.
\end{lemma}
\begin{proof}
    As $C$ is full-dimensional, we can find some $r_1,\dots,r_n\in C\cap\ZZ^n$ that are linearly independent.
    Reordering the indices if necessary, we may assume $\det(R)>0$.
    If $r_1,\dots,r_n$ does not form a basis of $\ZZ^n$, then we know that the matrix $R:=(r_1,\dots,r_n)\in\ZZ^n$ has $\det(R)>1$.
    By~\cite[Lemma 2.3.14]{deloera2012algebraic}, we know that there exists a nonzero $r'\in\ZZ^n$ in the fundamental parallelepiped $\{\sum_{i=1}^{n}\lambda_ir_i:0\le\lambda_i<1,\ i=1,\dots,n\}$, and thus if we replace $r_1$ with $r'$, we would have $\det(r',r_2,\dots,r_n)=\lambda_1\det(R)<\det(R)$.
    Since the determinant takes integer values, by repeating this procedure finitely many times, we may assume without loss of generality that $\det(R)=1$, i.e., $r_1,\dots,r_n$ becomes a basis for $\ZZ^n$.
    
    Now consider the following linear equations for $A\in\RR^{n\times n}$
    \[
        AR=(T\cdot r_1,\dots,T\cdot r_n).
    \]
    The right-hand side is an integer matrix because $T(S_C)\subseteq S_C$.
    This uniquely determines a solution $A=A_T\in\ZZ^{n\times n}$ as $\det(R)=1$.
    For any $x=\sum_{i=1}^{n}\mu_ir_i\in C$ where $\mu_1,\dots,\mu_n\in\RR$, we have that $A_Tx=\sum_{i=1}^{n}\mu_iT\cdot r_i=T\cdot x$ by the linearity of $T$ acting on $C$ (in Definition~\ref{def:RG-FG}).
    The unimodularity of $A_T$ follows from $A_T(T^{-1}\cdot r_1, \dots, T^{-1}\cdot r_n) = R$.
\end{proof}

\begin{remark}
    The system of linear equations $Ar_i=T\cdot r_i$, $i=1,\dots,n$ has a unique solution by the linear independence of the vectors $r_1,\dots,r_n$.
    As an illustration, note that when $n=2$, the system can be written as
    \[
    \begin{pmatrix}
        r_{11} & r_{12} & 0 & 0 \\
        0 & 0 & r_{11} & r_{12} \\
        r_{21} & r_{22} & 0 & 0 \\
        0 & 0 & r_{21} & r_{22}
    \end{pmatrix}
    \begin{pmatrix}
        a_{11} \\ a_{12} \\ a_{21} \\ a_{22}   
    \end{pmatrix}
    =\begin{pmatrix}
        (T\cdot r_1)_1 \\ (T\cdot r_1)_2 \\ (T\cdot r_2)_1 \\ (T\cdot r_2)_2
    \end{pmatrix},
    \]
    which is clearly nondegenerate by the linear independence of $r_1$ and $r_2$.
\end{remark}

Next we point out a simple yet useful fact that will be used later in characterization of $(R,G)$-finitely generated cones.

\begin{lemma}\label{lem:extreme_rays}
    If $S_C$ is $(R,G)$-finitely generated, then for any $T\in G$ and any extreme ray $[r]\in C$, $[T\cdot r]$ is an extreme ray of $C$.
\end{lemma}
\begin{proof}
    Assume for contradiction that $[T\cdot r]$ is not an extreme ray.
    Then we can find linearly independent $r_1,\dots,r_M\in C$ such that $T\cdot r=\sum_{i=1}^{M}r_i$ for some $M>1$.
    Since $r=T^{-1}\cdot (T\cdot r)=\sum_{i=1}^{M}T^{-1}\cdot r_i$ defines an extreme ray $[r]$ of $C$, this implies that there exist $\mu_1,\dots,\mu_M>0$ such that $T^{-1}\cdot r_i=\mu_ir$. 
    Then $r_i=\mu_i T\cdot r$ for each $i=1,\dots,M$, which contradicts with the linear independence of $r_1,\dots,r_M$.
\end{proof}

We also prove that any action must fix the lineality space of $C$.
\begin{lemma}\label{lem:subspace}
    If $S_C$ is $(R,G)$-finitely generated and $L$ is the maximal subspace contained in $C$, then for any $T\in G$, $T\cdot L = L$.
\end{lemma}
\begin{proof}
    By the assumption of $L$, we have $C = L + C'$, where $C'$ is a pointed cone and $L\cap C'=\{0\}$.
    Assume, for the sake of contradiction, that there exists $T\in G$ and $x\in L$ such that $T\cdot x\notin L$.
    By Lemma \ref{lem:unimodular}, we know that $T\cdot x = A_T x$ for any $x\in C$.
    Then we can assume that $T\cdot x = A_T x = x_L + r$, where $x_L\in L$ and $r\in C'$. Because $-x\in L$, we have $T\cdot (-x) = A_T (-x) = -x_L - r = x_L' + r'$, where $x_L'\in L$ and $r'\in C'$. Thus, $r' + r = x_L' - x_L\in L$ but $r' + r \in C'$. Hence, $r' + r = 0$, a contradiction to $C'$ is pointed.
\end{proof}

For nontrivial full-dimensional cones, scaling actions cannot be included in the group $G$.
\begin{lemma}\label{lem:no_scaling}
    Suppose that $C\subsetneq\RR^n$ is full dimensional, and $S_C$ is $(R,G)$-finitely generated. For any $T\in G$, if $A_T=\lambda I_n$, then $\lambda = 1$.
\end{lemma}
\begin{proof}
    By~\Cref{lem:unimodular}, $A_T$ must be unimodular, so its determinant $\det(A_T)=\lambda^n=1$.
    Thus $\lambda=\pm1$.
    The case $\lambda=-1$ is impossible: take any $n$-ball $B\subset C$ and $T\cdot B=\lambda B\subset C$ by the definition of $(R,G)$-finite generation.
    This means that both $B$ and $-B$ are contained in $C$, and thus $C=\RR^n$ which is a contradiction.
\end{proof}

\begin{lemma}\label{lem:finite_G}
Suppose that $S_C$ is $(R,G)$-finitely generated. If $G$ is finite, then $C$ is a rational polyhedral cone.
\end{lemma}
\begin{proof}
Suppose that $S_C$ is $(R,G)$-finitely generated with finite $G$. Then we know that the generating set $R' = \{g\cdot r: g\in G, r\in R\}$ is also finite. Therefore, $S_C$ is finitely generated by $R'$, which implies that $C$ is a rational polyhedral cone.
\end{proof}

The following fact from linear algebra will also be useful.
\begin{lemma}\label{lemma:RationalEigenvalue}
    Let $u\in\QQ^n$ and matrix $A\in\GL(n,\ZZ)$ such that $Au=\lambda u$ for some $\lambda\in\RR$.
    Then $\lambda=1$.
\end{lemma}
\begin{proof}
    By assumption, there exists $\lambda>0$ such that $Au=\lambda u$.
    Since $A\in\GL(n,\ZZ)$ and $u\in\QQ^n$, $\lambda\in\QQ$.
    Consider the characteristic polynomial $p_A(t)\in\ZZ[t]$ of $A$, which is divisible by $t-\lambda$ in $\QQ[t]$ and primitive due to $\det(A)=1$.
    By Gauss' lemma, we know that $p_A(t)$ is reducible over $\ZZ$.
    In fact, suppose it factors as $p_A=q_1q_2\in\ZZ[t]$, in which case the constants $q_1(0)=q_2(0)=\pm1$ so they are also primitive in $\ZZ[t]$.
    Since either $q_1$ or $q_2$ must be divisible by $t-\lambda$ in $\QQ[t]$, by repeating the argument, we must have $t-\lambda\in\ZZ[t]$, or equivalently, $\lambda\in\ZZ$ and thus $\lambda=1$.
\end{proof}

\subsection{Non-examples of \texorpdfstring{$(R,G)$}{(R,G)}-finite generation}
\label{sec:nonpolyhedral}

Now we are ready to show Theorem \ref{thm:FiniteRationalRays} that states an $(R,G)$-finitely generated nonpolyhedral cone must have infinitely many rational extreme rays, or all of its rational extreme rays lie in a proper subspace.

\begin{proof}[Proof of Theorem \ref{thm:FiniteRationalRays}]
    By scaling we may assume that $u_1,\dots,u_m\in\ZZ^n$.
    Suppose $C\cap\ZZ^n$ is $(R,G)$-finitely generated for some $G\subset\GL(n,\ZZ)$.
    From Lemma~\ref{lem:extreme_rays}, any $T\in G$ induces a permutation on the finite set of integral extreme rays $[u_1],\dots,[u_m]$.
    Then for each $T\in G$, $T^{m!}$ induces the identity permutation on these extreme rays.
    Since $C$ is full-dimensional, by~\Cref{lem:unimodular} there exists $A_T\in\GL(n,\ZZ)$ representing $T$, so $A_T^{m!}\in\GL(n,\ZZ)$ represents $T^{m!}$.
    Then $u_1,\dots,u_m$ are eigenvectors of the matrix $A_T^{m!}$ since $[A_T^{m!}u_i]=[u_i]$ for $i=1,\dots,m$.
    By~\Cref{lemma:RationalEigenvalue}, their eigenvalues are all 1.
    Then $A_T^{m!}$ must be the identity matrix following the assumption that $u_1,\dots,u_m$ span $\RR^n$. 
    Thus from the bounded Burnside problem for linear groups~\cite[Theorem 6.13]{ceccherini2021burnside}, we know that $G$ must be finite.
    By Lemma~\ref{lem:finite_G}, $C$ must be a rational polyhedral cone, which is a contradiction.
\end{proof}

We are also ready to prove \Cref{thm:SmoothHypersurfaceCone}.
\begin{proof}[Proof of Theorem~\ref{thm:SmoothHypersurfaceCone}]
    Suppose $C\cap\ZZ^n$ is $(R,G)$-finitely generated.
    Since $C$ is full-dimensional, by Lemma~\ref{lem:unimodular} any action in $G$ is represented by a unimodular matrix $A\in\GL(n,\ZZ)$.
    Such $A$ must preserve the boundary of $C$ in $\RR^n$, and thus define an automorphism of $\PP^{n-1}$ that preserves the algebraic boundary of $C$ in $\PP^{n-1}$.
    By assumption, the algebraic boundary is a smooth hypersurface $H$ of degree $d$.
    Since the automorphism group of $H$ is finite for any $n\ge3$, $d\ge4$, and $(n,d)\neq(4,4)$~\cite{matsumura1963automorphisms}, our group $G$ must also be finite.
    Hence~\Cref{lem:finite_G} dictates $C$ to be a rational polyhedral cone, which contradicts with the assumption of smooth algebraic boundary.
\end{proof}

\section{Polyhedral cones and \texorpdfstring{$(R,G)$}{(R,G)}-finite generation}
\label{sec:main-polyhedral}

In this section, we provide the proofs for~\Cref{thm:polyhedral-necessary,thm:simple-sufficient} and use them to fully characterize the $(R,G)$-finitely generated cones in the plane, and all 3-dimensional pointed polyhedral cones.

\subsection{Properties of polyhedral cones}

We are now ready to prove~\Cref{thm:polyhedral-necessary}.
\begin{proof}[Proof of~\Cref{thm:polyhedral-necessary}]
    The case where $[u_1],\dots,[u_m]$ are rational is well-known so we assume that one of them is irrational.
    By Lemma~\ref{lem:finite_G}, we know that $G$ must be infinite in this case.
    From Lemma~\ref{lem:extreme_rays}, any $T\in G$ defines a permutation $\pi_T$ on the finite set of extreme rays $E:=\{[u_1],\dots,[u_m]\}$.
    Thus there must be at least two different actions $T,T'\in G$ such that they define the same permutation $\pi_{T}=\pi_{T'}$ on $E$.
    In particular, they have the same order $p\ge1$, i.e., $\pi_{T^p}=\pi_{(T')^p}$ is the identity permutation on $E$.

    If $T^p$ is not the identity in $G$, then let $A=A_{T^p}\in\GL(n,\ZZ)$ be the non-identity unimodular matrix defined in Lemma~\ref{lem:unimodular}.
    Since $[A u_i]=[u_i]$, there exists $\lambda_i>0$ such that $A u_i=\lambda_i u_i$, for each $i=1,\dots,m$.
    This shows that $u_1,\dots,u_m$ are eigenvectors of a desired $A$ with positive eigenvalues $\lambda_1,\dots,\lambda_m$.

    If $T^p$ is the identity in $G$, then $T'':=T^{p-1}T'$ is not the identity because $T'\neq T$.
    Now let $A=A_{T''}$ be the integer unimodular matrix associated with $T''$.
    By the same argument, we see that there exist $\lambda_1,\dots,\lambda_m>0$ such that $A u_i=\lambda_i u_i$, which shows that $u_1,\dots,u_m$ are eigenvectors of $A$ as desired.
\end{proof}

\begin{corollary}\label{cor:3DSimple}
    If $C\subset\RR^3$ is a pointed irrational polyhedral cone, and $S_C$ is $(R,G)$-finitely generated, then $C$ is simple.
\end{corollary}
\begin{proof}
    Suppose that $C$ is not simple.
    Then we can find at least four extreme rays $[u_1],\dots,[u_4]$ of $C$.
    By Theorem~\ref{thm:polyhedral-necessary}, $u_1,\dots,u_4$ must be eigenvectors of a non-identity matrix $A\in\GL(3,\ZZ)$.
    Two of them, say $u_1$ and $u_2$, must share the same eigenvalue.
    Then $u_3$ or $u_4$ cannot lie in $\operatorname{span}\{u_1,u_2\}$ because they define extreme rays.
    As the sum of geometric multiplicities being at most 3, $u_3$ and $u_4$ must be collinear, which is a contradiction to them being extreme rays.
\end{proof}

Given a polyhedral cone $C\subset\RR^n$ with extreme rays $[u_1],\dots,[u_m]$, let $H_C\subset\GL(n,\ZZ)$ be the set of unimodular matrices that have $u_1,\dots,u_m$ as eigenvectors with positive eigenvalues.
Here, the eigenvalues corresponding to rational extreme rays must be 1 by~\Cref{lemma:RationalEigenvalue}.
For any group $G\subset\GL(n,\ZZ)$ fixing the polyhedral cone $C$, $H:=G\cap H_C$ is a normal subgroup of $G$ and determines whether $C$ is $(R,G)$-finitely generated for some finite $R\subset S_C$.
\begin{lemma}\label{lemma:NormalSubgroup}
    Let $C\subset\RR^n$ be a polyhedral cone.
    If $S_C$ is $(R,G)$-finitely generated, then it is also $(R_H,H)$-finitely generated for the normal subgroup $H:= G\cap H_C$ and some finite $R_H\subset S_C$.
\end{lemma}
\begin{proof}
    We first note that $H\subseteq G$ is a normal subgroup and each coset of $G/H$ corresponds to a permutation of the extreme rays $[u_1],\dots,[u_m]$ of $C$.
    To see this, take any $A\in H$ and $B\in G$.
    For each extreme ray $[u_i]$, $i=1,\dots,m$, of $C$, by Lemma~\ref{lem:extreme_rays}, there exists $\mu_i>0$ such that $Bu_i=\mu_iu_j$ for some $j=1,\dots,m$.
    By assumption, $Au_j=\lambda_ju_j$ for some $\lambda_j>0$, which implies that $B^{-1}ABu_i=\mu_i^{-1}\lambda_j\mu_iu_i=\lambda_ju_i$ and thus $B^{-1}AB\in H$, proving the normality of $H$.
    Now suppose $B,B'\in G$ induce the same permutation on the set of extreme rays, then for any $u_i$, $B^{-1}B'[u_i]=[u_i]$ and thus $B^{-1}B'\in H$.
    This shows that each coset of $H$ in $G$ corresponds to a permutation of extreme rays, which further shows that $G/H$ is finite.

    Now suppose $S_C$ is $(R,G)$-finitely generated.
    Pick $A_1,\dots,A_k\in G$ representing distinct cosets of $H$ in $G$, for some $k\in\ZZ_{\ge0}$, and define $R_H:=\cup_{i=1}^{k}A_iR$.
    For any $s\in S_C$, by definition, there exists $B_1,\dots,B_l\in G$, $u_1,\dots,u_l\in R$, and $\lambda_1,\dots,\lambda_l\in\ZZ_{\ge0}$, such that $s=\sum_{j=1}^{l}\lambda_jB_ju_j$.
    Since $B_j=B'_jA_{i(j)}$ for some $i(j)\in\{1,\dots,k\}$ and some $B'_j\in H$, the above sum can be rewritten as
    \[
        s =\sum_{j=1}^{l}\lambda_jB'_j(A_{i(j)}u_j),
    \]
    completing the proof as each $A_{i(j)}u_j\in R_H$.
\end{proof}

The above lemma allows us to reduce the question of identifying groups $G$ for $(R,G)$-finite generation to subgroups of $H_C$, i.e., unimodular matrices for which the extreme rays of $C$ are eigenvectors with positive eigenvalues.
Next we will study this group by associating it with the algebraic number fields of the extreme rays, and propose a sufficient condition for $(R,G)$-finite generation of simple polyhedral cones.

\subsection{Sufficient condition for simple polyhedral cones}

Given a simple polyhedral cone $C\subset\RR^n$ with linearly independent extreme rays $[u_1],\dots,[u_n]$, let $E_C\subset\RR^{n\times n}$ denote the subspace of matrices that have $u_1,\dots,u_n$ as its eigenvectors.
Our plan to prove~\Cref{thm:simple-sufficient} consists of the following two steps:
\begin{enumerate}
    \item use the existence of a matrix satisfying the sufficient condition to show that there are sufficiently many matrices in $H_C$, and
    \item leverage the difference of the eigenvalues of matrices in $H_C$ to show that all integer points in $C$ can be generated by a finite subset.
\end{enumerate}
To begin with, we show that the sufficient condition is equivalent to the subspace $E_C$ being rational.
For notational convenience, we denote $E_C(\QQ):= E_C\cap\QQ^{n\times n}$ as the rational matrices in $E_C$.
Recall that a \emph{primary rational canonical form} of a rational matrix $A\in\QQ^{n\times n}$ is a block-diagonal matrix $B\in\QQ^{n\times n}$ that is similar to $A$ over $\QQ$, where each block is a companion matrix corresponding to an irreducible factor of the characteristic polynomial of $A$~\cite[Chapter VII, Corollary 4.7(ii)]{hungerford2012algebra}.

\begin{lemma}\label{lemma:RationalSubspace}
    Suppose $C$ is a simple cone with extreme rays $[u_1],\dots,[u_n]$ where $[u_1],\dots,[u_k]$ are rational.
    The following three conditions are equivalent.
    \begin{enumerate}
        \item The subspace $E_C$ is rational, i.e., $\dim_\QQ(E_C(\QQ))=n$.
        \item There exist $A\in E_C(\QQ)$ and $\lambda_1,\dots,\lambda_n\in\RR_{>0}$ such that $Au_i=\lambda_i u_i$ for each $i=1,\dots,n$ and $\lambda_{k+1},\dots,\lambda_n$ are distinct.
        \item There exist $A\in E_C(\QQ)$ and all distinct $\lambda_1,\dots,\lambda_n\in\RR_{>0}$ such that $Au_i=\lambda_i u_i$ for each $i=1,\dots,n$.
    \end{enumerate}
\end{lemma}
\begin{proof}
    For 1$\implies$2, consider the linear map $\mu\in\RR^n\mapsto U\Diag(\mu)U^{-1}$, where $U=(u_1,\dots,u_n)\in\RR^{n\times n}$ and $\Diag(\mu)$ is the diagonal matrix constructed from $\mu$.
    By definition of eigenvectors, this is an isomorphism between $\RR^n$ and $E_C$.
    Thus if $E_C$ is rational, all vectors $\mu$ such that $U\Diag(\mu)U^{-1}\in\QQ^{n\times n}$ form a dense subset in $\RR^n$, from which we conclude that there exists $\lambda\in\RR^n_{>0}$ with distinct components such that $A=U\Diag(\lambda)U^{-1}\in\QQ^{n\times n}$.

    For 2$\implies$3, consider the \emph{primary rational canonical form} of $A$: there exists an invertible matrix $V=(v_1,\dots,v_n)\in\QQ^{n\times n}$ such that $A=VBV^{-1}$, where $B\in\QQ^{n\times n}$ is a block-diagonal matrix with the top-left $k\times k$ submatrix being a diagonal matrix consisting of the eigenvalues $\lambda_1,\dots,\lambda_k\in\QQ$.
    Thus $v_j=u_j\in\QQ^n$ for $j=1,\dots,k$, and $\operatorname{span}_\RR\{v_{k+1},\dots,v_n\}=\operatorname{span}_\RR\{u_{k+1},\dots,u_n\}=:L\subset\RR^n$.
    A perturbation $B':=B+\Diag(\epsilon_1,\dots,\epsilon_k,0,\dots,0)$ by some sufficiently small $\epsilon_1,\dots,\epsilon_k\in\QQ$ gives a matrix $A':=VB'V^{-1}\in\QQ^{n\times n}$ with all distinct eigenvalues. 
    The matrix $A'\in E_C$ because for any $j=1,\dots,k$, $A'u_j=A'v_j=(\lambda_j+\epsilon_j)u_j$ and its restriction to the subspace $L$ satisfies $A'\vert_L=A\vert_L$ so $A'u_j=\lambda_ju_j$ for each $j=k+1,\dots,n$.
    
    For 3$\implies$1, consider the subspace $H$ spanned by matrices $I,A,A^2,\dots,A^{n-1}$ over $\QQ$.
    Clearly $H\subseteq E_C(\QQ)$ so it suffices to show that $\dim_\QQ(H)=n$.
    Assume for contradiction that there exist $c_0,\dots,c_{n-1}\in\QQ$ that are not all zero, such that $\sum_{i=0}^{n-1}c_i A^i=0$.
    This implies that $\sum_{i=0}^{n-1}c_i\lambda_j^i=0$ for each $j=1,\dots,n$, which is a contradiction because the Vandermonde matrix
    \begin{equation*}
        \begin{pmatrix}
            1 & \lambda_1 & \lambda_1^2 & \cdots & \lambda_1^{n-1} \\
            1 & \lambda_2 & \lambda_2^2 & \cdots & \lambda_2^{n-1} \\
            \vdots & & & & \vdots \\
            1 & \lambda_n & \lambda_n^2 & \cdots & \lambda_n^{n-1} \\
        \end{pmatrix}
    \end{equation*}
    has full rank when $\lambda_1,\dots,\lambda_n$ are distinct.
\end{proof}

We thus call any matrix with all distinct eigenvalues (condition 3 in~\Cref{lemma:RationalSubspace}) in $E_C(\QQ)$ a \emph{generating matrix}.
The remainder of the first step of the proof is based on the following Dirichlet unit theorem, the proof of which follows from Proposition 6.10 of~\cite{jordan2019analytic}.
Here, we will take the product field $K$ to be $E_C(\QQ)$, the order $\mathcal{O}$ to be a subring generated by a generating matrix over $\ZZ$, and then the unit group $\mathcal{O}^\times$ would correspond exactly to group of unimodular matrices $H_C$ in our context.
\begin{proposition}\label{prop:DirichletUnitTheorem}
    Let $K=K_1\times\cdots\times K_m$ be a product of number fields $K_1,\dots,K_m$, and $\mathcal{O}\subseteq K$ be an order, i.e., a subring of $K$ finitely generated as a $\ZZ$-module such that $\QQ\mathcal{O}=K$.
    Suppose $K\otimes\RR\cong\RR^{r}$.
    Then the unit group $\mathcal{O}^\times$ is a finitely generated Abelian group of rank $r-m$.
\end{proposition}

Next we set up the second step of the proof for~\Cref{thm:simple-sufficient}.
We introduce the following notation that can be used to characterize the size of $H_C$ for a simple cone $C\subset\RR^n$.
Fix the order of the extreme ray vectors $[u_1],\dots,[u_n]$. 
Then for any $A\in H_C$, we use $\lambda_{C}(A)\in\RR_{>0}^n$ to denote the vector of eigenvalues corresponding to $u_1,\dots,u_n$.
Any subgroup $G\subset H_C$ then acts on $C$ through componentwise multiplication of eigenvectors, i.e., for any $A\in G$ and $x=\sum_{i=1}^{n}x_iu_i$, $(x_1,\dots,x_n)\in\RR_{\ge0}^n$, $Ax=\sum_{i=1}^{n}x_i(Au_i)=\sum_{i=1}^{n}\lambda_C(A)_i x_iu_i$.
To describe the relation among elements of $G$, we say $A_1,\dots,A_k$ have \emph{log-linearly independent} eigenvalues for some $k\in\ZZ_{\ge1}$ if the componentwise logarithm vectors $\log(\lambda_C(A_1)),\dots,\log(\lambda_C(A_k))\in\RR^n$ are linearly independent.

\begin{lemma}\label{lemma:HighDimBalancing}
    Let $C\subset\RR^n$ be a simple polyhedral cone and $F\subseteq C$ be a $d$-dimensional face of $C$ with $d\le n$.
    Suppose $G\subseteq H_C$ a group that contains $d-1$ log-linearly independent matrices $A_1,\dots,A_{d-1}$ such that $A_iu=u$ for any extreme ray $u\in C\setminus F$. 
    Then there exists a rational polyhedral cone $P\subset F$ such that for any $x\in F\cap\ZZ^n$, there exists $A\in G$, $Ax\in P$.
\end{lemma}
\begin{proof}
    Let $u_1,\dots,u_d\in\RR^n$ denote the extreme rays of $F$.
    Take any subset $J\subset[n]:=\{1,\dots,n\}$ such that the cone $F_J:=\{\sum_{i\in J}c_iu_i:c_i>0,\,i\in J\}$, which is the relative interior of a face of $F$, satisfies $F_J\cap\QQ^n\neq\varnothing.$
    We claim that there exists a rational polyhedral cone $P_J\subset F_J$ such that for any $x\in F_J\cap\ZZ^n$, we can find $A\in G$ with $Ax\in P_J$.
    The assertion then follows by taking the convex hull of the union of such rational polyhedral cones $P:=\conv(\cup_{J\in 2^{[n]}}P_J)\subset F$.

    By assumption, let $A_1,\dots,A_{d-1}\in G$ be matrices such that $\log(\lambda_C(A_1)),\dots,\log(\lambda_C(A_{d-1}))\in\RR^n$ are linearly independent. 
    Moreover, let $A_0:=aI$ be a scaled identity matrix with $a>1$, so $\log(\lambda_C(A_0))=\log(a)(1,\dots,1)$ and thus perpendicular to $\log(\lambda_C(A_i))$ because $\det(A_i)=1$ for any $i=1,\dots,n-1$.
    Let $b_0,\dots,b_{d-1}$ be the restrictions of $\log(\lambda_C(A_0)),\dots,\log(\lambda_C(A_{d-1}))$ to their coordinates with indices in $J$, which by the definition above are linearly independent, so they generate a full dimensional lattice $L\subset\RR^m$, $m=|J|$.
    We denote the closure of the fundamental parallelepiped of $L$ as $B\subset\RR^{m}$.
    Since $B$ is compact, the image of $B$ under the continuous map $\eta:(t_i)_{i\in J}\mapsto\sum_{i\in J}\exp(t_i)u_i$ is also compact.
    Then both the set $\eta(B)\cap\QQ^n\subset F_J$ and its convex hull are bounded. 
    Thus by the denseness of rational points in the subspace $\operatorname{span}(F_J\cap\QQ^n)$, there exists a rational polyhedral cone $P_J\subset F_J$ such that $\eta(B)\cap\QQ^n\subset P_J$.
    This implies that for any $x=\sum_{i\in J}x_iu_i\in F_J\cap\ZZ^n$, there exists a vector $b=(\beta_i)_{i\in J}\in L$ such that $(\log(x_i)+\beta_i)_{i\in J}\in B$. 
    Consequently, there is a matrix $A\in G$ and $k\in\ZZ$ such that $a^kAx\in \eta(B)\cap\ZZ^n\subset P_J$, and thus $Ax\in P_J$. 
\end{proof}

An illustration of the set $\eta(B)$ in the above proof is presented in Figure~\ref{fig:mainfigure} (with construction in Example~\ref{ex:SOCSlice}).
We next extend the argument to partitions of extreme rays of the simple cone.
To simplify the notation, for a subset $J\subseteq[n]:=\{1,\dots,n\}$, let $H_C(J)\subset\GL(n,\ZZ)$ denote all unimodular matrices $A$ such that $Au_i=u_i$ for any $i\notin J$ (i.e., the eigenvalues of $A$ corresponding to $u_i$, $i\notin J$ are all 1).

\begin{lemma}\label{lemma:ReducibleBalancing}
    Let $C\subset\RR^n$ be a simple cone with extreme rays $[u_1],\dots,[u_n]$ for some $u_1,\dots,u_n\in\RR^n$.
    If there exists a partition $J_1,\dots,J_k$ of $[n]$ such that for each $i=1,\dots,k$, there are $l_i:=\vert J_i\vert-1$ matrices $A_{i1},\dots,A_{il_i}\in H_C(J_i)$, the vectors of eigenvalues of which $\lambda_C(A_{i1}),\dots,\lambda_C(A_{il_i})$ are log-linearly independent, then $S_C$ is $(R,G)$-finitely generated for the group $G\subset\GL(n,\ZZ)$ generated by $\cup_{i=1}^{k}\{A_{i1},\dots,A_{il_i}\}$.
\end{lemma}
\begin{proof}
    Let $G\subset\GL(n,\ZZ)$ be the subgroup generated by $\cup_{i=1}^{k}\{A_{i1},\dots,A_{il_i}\}$.
    We claim that there exists a rational polyhedral cone $P\subset C$ such that for any $x\in C\cap\ZZ^n$, there exists $A\in G$ such that $Ax\in P$.
    Given this claim, it is easy to see that $S_C=C\cap\ZZ^n$ is $(R,G)$-finitely generated where $R$ can be taken to be a Hilbert basis of $P\cap\ZZ^n$.
    
    We prove this claim by induction on the size of the partition $k$.
    When $k=1$, $l_1=n-1$, and this is the case proved in Lemma~\ref{lemma:HighDimBalancing}.
    Now assume that the lemma is true for any $k-1$ partitions, and we consider the case $k\ge2$.
    Let $J'=\cup_{i=1}^{k-1}J_i$, $C':=\{\sum_{j\in J'}x_ju_j:x_j\ge0,\,j\in J'\}$, and $G'$ be the group generated by $\cup_{i=1}^{k-1}\{A_{i1},\dots,A_{il_i}\}$.
    By the induction hypothesis, there exists a polyhedral cone $P'\subset C'$, such that for any $x'\in C'\cap\ZZ^n$, there exists $A'\in G'$ such that $A'x\in P'$.
    Again by Lemma~\ref{lemma:HighDimBalancing}, there exists a rational polyhedral cone $P''\subset C'':=\{\sum_{j\in J_k}x_ju_j:x_j\ge0,\,j\in J_k\}$ such that for any $x''\in C''$, there exists a matrix $A''\in\langle A_{k1},\dots,A_{kl_k}\rangle$ such that $A''x''\in P''$.
    We define $P=\conv(P'\cup P'')$, and for any $x\in C\cap\ZZ^n$, let $A'$ and $A''$ be the corresponding matrices defined above, and $A:=A'A''$.
    Since we can write $x=x'+x''$ where $x'\in C'$ and $x''\in C''$, due to the simplicity of $C$, we have $Ax=A'A''x'+A''A'x''=A'x'+A''x''\in P'+P''\subset P$ because $A''x'=x'$ and $A'x''=x''$ by our assumption and induction hypothesis.
\end{proof}

Now we are ready to prove Theorem~\ref{thm:simple-sufficient}. 
\begin{proof}[Proof of Theorem~\ref{thm:simple-sufficient}]
    By~\Cref{lemma:RationalSubspace}, we can take $A$ to be a generating matrix of $E_C(\QQ)$ with distinct eigenvalues, and further assume $A\in\ZZ^{n\times n}$ through scaling.
    Let $p_A(t)\in\ZZ[t]$ denote the characteristic polynomial of $A$, and suppose it factors into $p_A(t)=\prod_{i=1}^{m}q_{A,i}(t)$ over $\QQ$ with $d_i:=\deg(q_{A,i})$ for $i=1,\dots,m$.
    Here $q_{A,1},\dots,q_{A,m}$ must be pairwise comaximal since they do not share roots in $\RR[t]$.
    Then $E_C(\QQ)$ has a $\QQ[t]$-module structure in the following way: any $g\in\QQ[t]$ acts on $M\in E_C(\QQ)$ by $g(t)\cdot M=g(A)M$.
    \Cref{lemma:RationalSubspace} asserts that $E_C(\QQ)$ is a cyclic module, and is thus isomorphic to $\QQ[t]/(p_A(t))$, which is further isomorphic to $\prod_{i=1}^{m}\QQ[t]/(q_{A,i}(t))$ by the Chinese remainder theorem.
    In this way, we can view $E_C(\QQ)$ as a product of algebraic number fields $\QQ[t]/(q_{A,i}(t))\cong\QQ(\alpha_i)$, for some root $\alpha_i\in\RR$ of $q_{A,i}(t)$.

    Consider the subring $\mathcal{O}\subset E_C(\QQ)$ generated by $\ZZ[t]\subset\QQ[t]$ acting on the $n\times n$ identity matrix $I\in E_C(\QQ)$, which is finitely generated by $I,A,\dots,A^{n-1}\in E_C(\QQ)$ as a $\ZZ$-module and $\QQ\mathcal{O}=E_C(\QQ)$ by~\Cref{lemma:RationalSubspace}.
    Thus $\mathcal{O}$ is an order of $E_C(\QQ)$.
    By Dirichlet's unit theorem (\Cref{prop:DirichletUnitTheorem}), the unit group $\mathcal{O}^\times$ of $\mathcal{O}$ is a finitely generated Abelian group of rank $n-m$.
    Let $\mathcal{O}_{\max}$ be the maximal order of $E_C(\QQ)$, which by the same theorem has a unit group of the same rank.
    This means $[\mathcal{O}_{\max}^\times:\mathcal{O}^\times]=:l<\infty$, and thus for each generator of $\mathcal{O}_{\max}^\times$, its $l$-th power lies in $\mathcal{O}^\times$.
    To be more specific, let $\{\beta_{i,j}\}_{j=1}^{d_i-1}$ be the $d_i-1$ generators of the unit group of the maximal order $\ZZ[\alpha_i]\subset\QQ(\alpha_i)$. 
    Then the vector $b_{i,j}:=(1,\dots,1,\beta_{i,j}^l,1,\dots,1)$ ($\beta_{i,j}^l$ appears on $i$-th coordinate and 1 on all others) lies in $\mathcal{O}^\times$.
    The vector $b_{i,j}$, viewed as an automorphism of $\prod_{i=1}^{m}\QQ(\alpha_i)$, defines a unimodular matrix $B_{i,j}$ in $\GL(n,\QQ)$, and is further integral since the order $\mathcal{O}$ is generated by integral matrices $I,A,\dots,A^{n-1}$.  
    In particular, note that $B_{i,j}\in H_C(J_i)$ by construction, where $J_i\subset[n]$ are the indices of eigenvectors corresponding to the factor $q_{A,i}$.
    Therefore, by Lemma~\ref{lemma:ReducibleBalancing}, $S_C$ is $(R,G)$-finitely generated where $G$ can be generated by $\{A_{i,1},\dots,A_{i,d_i-1}\}_{i=1}^{m}$.
\end{proof}

\subsection{\texorpdfstring{$(R,G)$}{(R,G)}-finitely generated cones in the plane}
\label{sec:2dim}

The cone in dimension $n=2$ is either a pointed simple cone or a half-space, so we want to fully characterize the $(R,G)$-finitely generated cones in dimension $n=2$.
In this subsection, we first prove a sufficient and necessary condition for $(R,G)$-finite generation in Theorem~\ref{thm:main2d}. We also show some examples, and prove all the possible groups $G$ for $(R,G)$-finitely generated cones.
Then we show that the $(R,G)$-finitely generated non-pointed cones in dimension two (i.e., half-spaces) can also be characterized similarly.

\subsubsection{Pointed cone case}
By Lemma~\ref{lem:no_scaling}, a non-identity matrix in $\mathrm{GL}(2,\ZZ)$ must have distinct eigenvalues, thus the necessary condition in Theorem~\ref{thm:polyhedral-necessary} is also sufficient.
\begin{theorem}\label{thm:main2d}
    Suppose $C\in\RR^2$ is a pointed convex cone with extreme rays $[u_1], [u_2]$, where $u_1, u_2\in\RR^2$. Then $S_C$ is $(R,G)$-finitely generated, if and only if either $[u_1], [u_2]$ are both rational; or $u_1, u_2$ are two eigenvectors with positive eigenvalues of a non-identity unimodular matrix $A\in\mathrm{GL}(2,\ZZ)$.
\end{theorem}

The example in \cite[Example 1]{deLoera2025integer} is a cone generated by irrational rays in the plane, from which one might want to attribute the $(R,G)$-finite generation with irrationality of the cone $C$.
We present in Example~\ref{ex:SOCSlice} another cone generated by irrational rays that has its conical semigroup to be $(R,G)$-finitely generated.

\begin{example}\label{ex:SOCSlice}
    Let $C=\text{cone}(u_1, u_2)$, where $u_1 = (1,\sqrt{2}), u_2 = (-1,\sqrt{2})$. 
    By Theorem~\ref{thm:main2d}, we know that $S_C$ is $(R,G)$-finitely generated because there exists $A = \begin{pmatrix} 3 & 2\\ 4 & 3\end{pmatrix}\in \mathrm{GL}(2, \ZZ)$ such that the two extreme rays of $C$ are eigenvectors of $A$ with positive eigenvalues. 

To find such a matrix $A$, let $A=\begin{pmatrix}
        a_{11} & a_{12}\\
        a_{21} & a_{22}
    \end{pmatrix}\in\mathrm{GL}(2,\ZZ)$ with eigenvalues $\lambda_1$, $\lambda_2$.
Then noting that $\lambda_1\lambda_2 = 1$ and $u_1, u_2$ are two eigenvectors of $A$, we must have
\begin{equation*}
\begin{cases}
    a_{11} = a_{22}, 2a_{12} = a_{21},\\
    a_{11}a_{22} - a_{12} a_{21} = 1.\\
\end{cases}
\end{equation*}
Therefore, we just find integer solutions to the Pell's equation $x^2 - 2y^2 = 1$ such that $x-\sqrt{2} y\ge 0$. A solution $(3, 2)$ will construct the given matrix $A$.

Once the group $G=\langle A\rangle$ is given, we want to use this as an example to show the construction of a rational polyhedral cone $P$ to obtain $R$ as in the proof of Lemma~\ref{lemma:HighDimBalancing}.
Because $[u_1], [u_2]$ are irrational, we only consider the interior of the cone $C$, which contains rational points. By construction, $A$ is a matrix with $\log(\lambda_C(A)) = (\log\lambda_1, \log\lambda_2)\ne 0$ and $\log\lambda_1 + \log\lambda_2=0$. Let $A_0:= aI$ with $\log(\lambda_C(A_0)) = (\log a, \log a)$ and $a>1$. Then the closure $B$ of the fundamental parallelepiped of the lattice $L$ generated by $\log(\lambda_C(A_0)),\log(\lambda_C(A))$ is $\{((\log\lambda_1) x_1 + (\log a) x_2, (\log\lambda_2) x_1 + (\log a) x_2): 0\le x_1\le 1, 0\le x_2\le 1\}$. Under the continuous map $\eta: (t_1, t_2)\mapsto (\exp{t_1}) u_1 + (\exp{t_2}) u_2$, $\eta(B) = \{\lambda_1^{x_1} a^{x_2} u_1 + \lambda_2^{x_1} a^{x_2} u_2: 0\le x_1\le 1, 0\le x_2\le 1\}\subset C$.
Because the rational points in the plane is dense, we can construct a rational polyhedral cone $P$ such that $\eta(B)\cap \QQ^2\subset P\subset C$.
One construction of $P$ is that $P=cone(t_1, t_2)$, where $t_1 = (0, 1)$, $t_2 = A t_1 = (2, 3)$.
Thus $S_C$ is $(R,G)$-finitely generated by $G=\left\langle A\right\rangle$ and $R$ that is the Hilbert basis of the rational cone spanned by $t_1, t_2$. We can also translate the fundamental parallelepiped to obtain a different construction of $P$.
\end{example}

\begin{figure}[H]
    \centering
    \subfigure{\includegraphics[height=4cm]{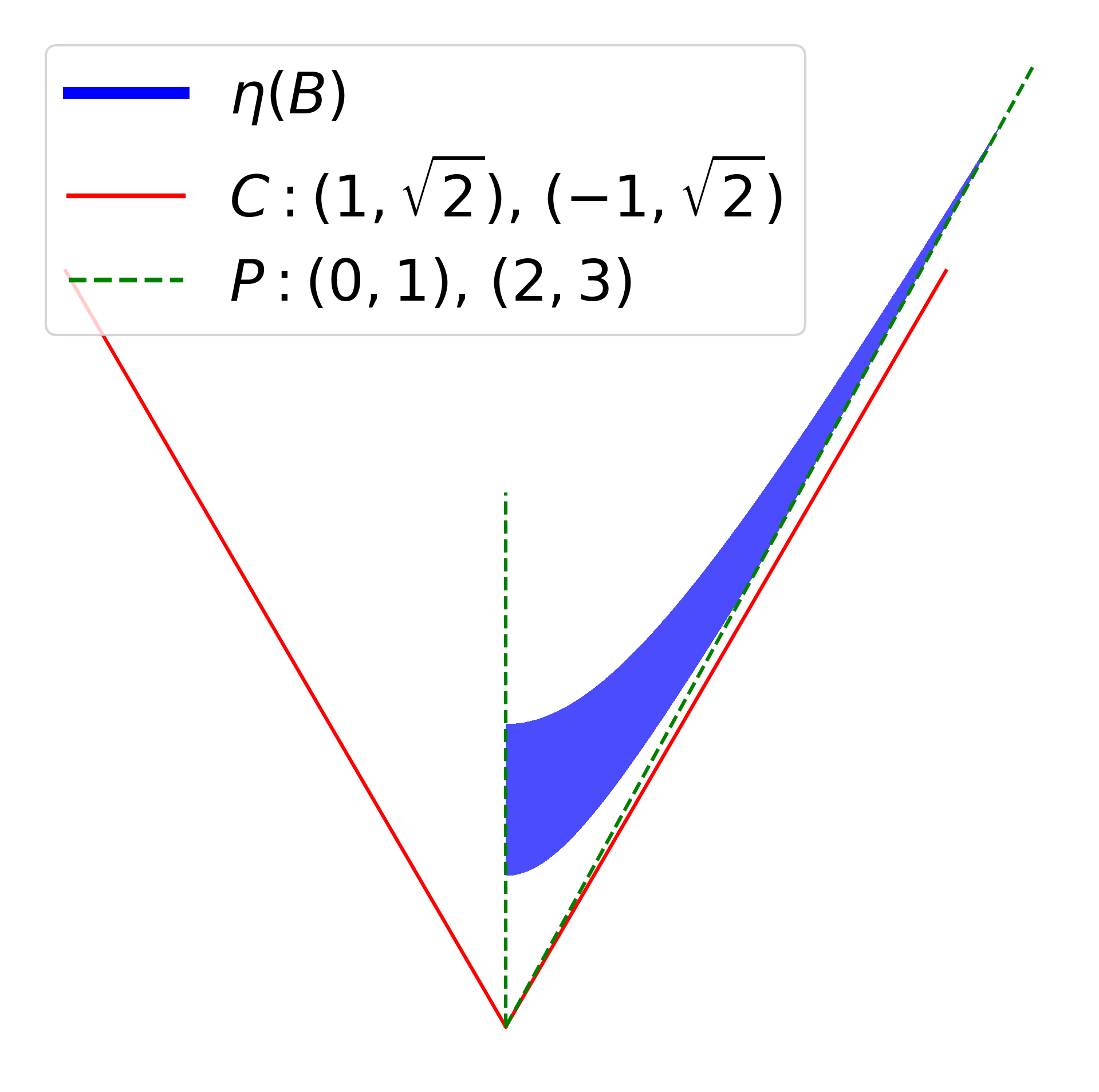}\label{fig:subfig1}}
    \hspace{2cm}
    \subfigure{\includegraphics[height=4cm]{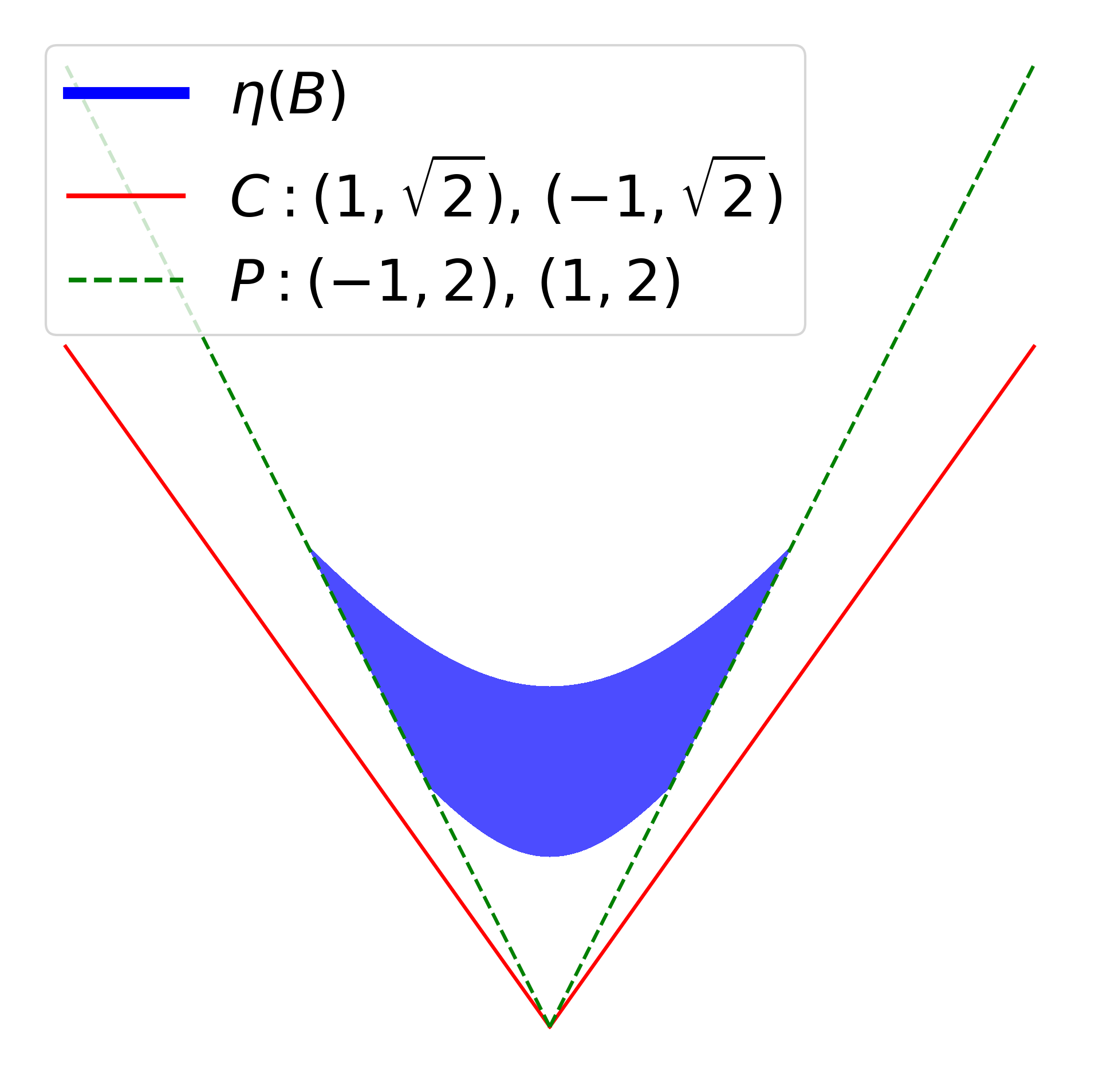}\label{fig:subfig2}}
    \caption{Two constructions of $P$ in Example \ref{ex:SOCSlice}.}
    \label{fig:mainfigure}
\end{figure}

Observe that $C$ in Example~\ref{ex:SOCSlice} is symmetric with respect to the $y$-axis. We can also choose
$$
G= \left\langle \begin{pmatrix} 3 & 2\\ 4 & 3\end{pmatrix},  
    \begin{pmatrix} -1 & 0\\ 0 & 1\end{pmatrix}\right\rangle.
$$
In the following theorem, we characterize all possible symmetric groups for $(R,G)$-finitely generated cones in dimension $n=2$.

\begin{theorem}\label{thm:group}
Suppose $C\in\RR^2$ is a pointed convex cone and let $G$ be the group of integer matrices fixing $C$ and $S_C$.
Then $G$ has four possibilities: 
\begin{enumerate}
    \item the trivial group (only identity),
    \item the finite group $\ZZ_2$ (generated by one matrix switching extreme rays),
    \item the infinite group $\ZZ$ (generated by one matrix fixing extreme rays),
    \item the infinite dihedral group (generated by two matrices, one switching extreme rays, and one fixing extreme rays).
\end{enumerate}
\end{theorem}

\begin{proof}
By Lemma \ref{lem:extreme_rays}, we know that for any $g\in G$, $g$ either fixes the two extreme rays of $C$, or switches them.
Let $H:=G\cap H_C$ be the subgroup of $G$ consisting of matrices fixing the extreme rays of $C$.

If there exists a non-identity matrix in $H$, then we claim that $H$ is isomorphic to $\ZZ$, i.e. there is essentially a unique matrix $A$ (can also take $A^{-1}$) that generates $H$. 
By Lemma \ref{lem:unimodular}, the two eigenvalues of $T$ in $H$ are $\lambda$ and $1/\lambda$. Denote $\lambda_1(T)=\max\{\lambda,1/\lambda\}\ge 1$. If $\lambda_1(T)=1$ then $T$ is the identity matrix. Let $\lambda_0=\inf\{\lambda_1(T):~T\in H, \lambda_1(T)>1\}$. We claim that there exists a matrix $A\in H$ such that $\lambda_1(A)=\lambda_0$. Suppose no such matrix exists, then there are matrices $T_i$ in $H$ for which $\lambda_i=\lambda_1(T_i)>1$ converge to some number $\lambda_0$. Then these matrices $T_i$ approach some matrix $A$ (the eigenvectors are fixed, and eigenvalues approach $\lambda_0$ and $1/\lambda_0$). However, $T_i$ are integer matrices and form a discrete set, a contradiction. Then it remains to show that $A$ generates all of $H$.
Suppose, for the sake of contradiction, there is $B\in H$ such that its largest eigenvalue is not $\lambda_0^k$ for some $k\in\ZZ$. Then it lies between $\lambda_0^k$ and $\lambda_0^{k+1}$ for some $k$. By considering $B A^{-k}$ or $B A^k$, we find a matrix whose largest eigenvalue is less than $\lambda_0$ and larger than 1, which is contradiction to the minimality of $\lambda_0$. Therefore, if $H$ is not trivial group, then $H$ is isomorphic to $\ZZ$.

If there exists a matrix $Q$ that switches two extreme rays $[u_1], [u_2]$. Then $Q^2=I_2$ by Lemma~\ref{lem:no_scaling}. Suppose that we have two such matrices $Q$ and $B$, i.e., $Qu_1 = \lambda u_2$, $Bu_1 = \mu u_2$, $\lambda\ne\mu$. 
Then $QB u_1 = (\mu/\lambda) u_1$, $QB u_2 = (\lambda/\mu) u_2$, which implies $QB\in H$. This shows that $Q$ and $H$ generate all of $G$, because other element $B$ that switches extreme rays can be generated by $B = Q^{-1} (QB)$.

Therefore, if $H$ is trivial and there is one matrix switching extreme rays, then $G$ is the finite group $\ZZ_2$. If $H=\langle A\rangle$ is not trivial and there is one matrix $Q$ switching extreme rays, then $G=\langle A, Q\rangle$ is the infinite dihedral group because $Q^2 = I_2$ and $Q A Q = A^{-1}$.
\end{proof}

Let $s_1, s_2$ ($s_1\ne s_2$) be the slopes of the two extreme rays of cone $C$.
By Theorem~\ref{thm:main2d}, we know that a non-identity matrix in $H$ exists if and only if
\begin{equation*}\label{eqn:cond1}
\begin{cases}
    s_1 + s_2 = -\frac{a_{11}-a_{22}}{a_{12}},\\
    s_1 s_2 = -\frac{a_{21}}{a_{12}},\\
\end{cases}
~\text{where}~A=\begin{pmatrix}
        a_{11} & a_{12}\\
        a_{21} & a_{22}
    \end{pmatrix}\in\mathrm{GL}(2,\ZZ).
\end{equation*}
For integer matrix switching extreme rays, it satisfies $Q^2= I_2$.
Such matrix is called an \emph{involutory} matrix, and is in the form 
$\begin{pmatrix}
    a & b\\
    c & -a
\end{pmatrix}$, where $a^2 + bc = 1$ with the eigenvalues 1 and -1.
We can easily verify that
\begin{itemize}
    \item If $s_1s_2 = 1$, then $a=0$, $b = c = 1$;
    \item If $s_1 + s_2\in\ZZ$, then $b=0$, $a=1$, $c=(s_1 + s_2)$;
    \item If $1/s_1 + 1/s_2 \in \ZZ$, then $c=0$, $a=1$, $b=-(1/s_1 + 1/s_2)$.
\end{itemize}

Example \ref{ex:SOCSlice} is an example of case 4. We give examples for the other three cases.
\begin{example} 
Let $u_1 = (2,1), u_2 = (3,5)$. The potential matrix switching two extreme rays is $\begin{pmatrix}
        \frac{13}{7} & -\frac{5}{7}\\
        \frac{24}{7} & -\frac{13}{7}
    \end{pmatrix}$ is rational, but not integer. The group $G$ is the trivial group.
    
\end{example}
\begin{example}
Let $u_1 = (1,1), u_2 = (-1,1)$. The matrix switching two extreme rays is $\begin{pmatrix}
        1 & 0\\
        0 & -1
    \end{pmatrix}$. The group $G$ is isomorphic to $\ZZ_2$.

\end{example}
\begin{example}
Let $u_1 = (1,\frac{8-3\sqrt{11}}{5}), u_2 =(1,\frac{8+3\sqrt{11}}{5})$. The matrix fixing two extreme rays is $A=\begin{pmatrix}
        2 & 5\\
        7 & 18
    \end{pmatrix}$, but no integer matrix switching two extreme rays. The group $G$ is generated by one matrix $A$, and is isomorphic to $\ZZ$.

\end{example}

\subsubsection{Half-space case}
Next, we consider the half-space case.

\begin{theorem}\label{thm:halfspace}
    Suppose $C\in\RR^2$ is a half-space associated with the line spanned by $r\in\RR^2$. Then $S_C$ is $(R,G)$-finitely generated, if and only if $r$ is rational or $r$ is the eigenvector with positive eigenvalue of a non-identity unimodular matrix $A\in\mathrm{GL}(2, \ZZ)$ with $\det(A)=1$.
\end{theorem}
\begin{proof}[Proof of Theorem \ref{thm:halfspace}]
For the necessary condition, by Lemma \ref{lem:subspace}, we know that for any $T\in G$, $T\cdot L = L$, where $L=\mathrm{span}(r)$. For any non-identity $T\in G$, we have $T\cdot r = \lambda r$ for some $\lambda\in \RR$.
By Lemma \ref{lem:unimodular}, we know that $T\cdot r = A_T r$, where $A_T\in \mathrm{GL}(2,\ZZ)$. Therefore, $r$ is an eigenvector of $A_T$.

If $\lambda>0$ and $\det(A_T)=1$, then $A=A_T$ satisfies the condition.

If $\lambda>0$ and $\det(A_T)=-1$, then $A_T$ has another eigenvector $r'\in C$ such that $A_T r' = \lambda' r'$, $\lambda'<0$, which contradicts $T\cdot C = C$.

If $\lambda<0$ and $A_T^2 \ne I_2$, then $A = A_T^2$ satisfies the condition, because $A_T^2 r = \lambda^2 r$ and $\det(A_T^2)=1$.

If $\lambda<0$, $A_T^2=I_2$ and $\det(A_T)=1$, then $A_T$ has another eigenvector $r'\in C$ such that $A_T r' = \lambda' r'$, $\lambda'=1/\lambda<0$, which contradicts $T\cdot C = C$.

Thus, we only need to consider the case where $\lambda<0$, $A_T^2=I_2$, $\det(A_T)=-1$ for all $T\in G$.
If there is only one such matrix, then $G$ is finite. By Lemma \ref{lem:finite_G}, we know that $r$ is rational. If there exist distinct $A_{T}$ and $A_{T'}$ such that $A_{T}^2= I_2$, $A_{T'}^2= I_2$, $A_{T} r=\lambda r$, $A_{T'} r= \lambda' r$, $\lambda,\lambda'<0$, $\det(A_T)=\det(A_{T'})=1$. Then we have $A_{T'}A_T\ne I_2$, $\det(A_TA_{T'})=1$ and $A_{T'}A_T r = \lambda\lambda' r$, where $\lambda\lambda'>0$. Therefore, $A=A_{T'}A_T$ satisfies the condition.

Therefore, $r$ is the eigenvector with positive eigenvalue of a non-identity unimodular matrix $A$.

The sufficient condition follows from Theorem~\ref{thm:simple-sufficient} after dividing the half-space into two pointed cones with extreme rays $r$, $r'$ and $-r$, $r'$, respectively, where $r'\in C$ is the other eigenvector of $A$.
\end{proof}

\subsection{Three-dimensional \texorpdfstring{$(R,G)$}{(R,G)}-finitely generated pointed polyhedral cones}
\label{sec:3dim}

In this section, we want to extend Theorem~\ref{thm:main2d} to 3-dimensional cones by examining the distinct eigenvalues.

\begin{theorem}\label{thm:main3d}
    Suppose $C\subset\RR^3$ is a polyhedral cone. Then $S_C$ is $(R,G)$-finitely generated, if and only if either all extreme rays are rational, or $C$ is simple with extreme rays $[u_1],[u_2],[u_3]$ that are eigenvectors with distinct positive eigenvalues of a non-identity unimodular matrix $A\in\operatorname{GL}(3,\ZZ)$.
\end{theorem}

\begin{proof}[Proof of Theorem~\ref{thm:main3d}]
    The sufficiency follows from~\Cref{thm:simple-sufficient}.
    For the necessity, if the extreme rays of $C$ are not all rational, then by by~\Cref{cor:3DSimple}, the only possibilities for an irrational polyhedral cone to be $(R,G)$-finitely generated is simple. Applying \Cref{thm:polyhedral-necessary} shows that $u_1, u_2, u_3$ are eigenvectors with positive eigenvalues of a non-identity unimodular matrix $A\in\operatorname{GL}(3,\ZZ)$. We show that the eigenvalues $\lambda_1,\lambda_2,\lambda_3$ of $A$ are distinct.

    If $\lambda_1, \lambda_2, \lambda_3\notin \QQ$, then the characteristic polynomial $p_A(t)\in\ZZ[t]$ is irreducible over $\QQ$, and all eigenvalues of $A$ are distinct.
    
    If only one of $\lambda_1,\lambda_2,\lambda_3$ is in $\QQ$, then by~\Cref{lemma:RationalEigenvalue}, $p_A(t)=(t-1)q(t)$ where $q(t)$ is an irreducible quadratic polynomial, the roots of which must be irrational and distinct. Thus, all eigenvalues of $A$ are distinct.

    If at least two of $\lambda_1,\lambda_2,\lambda_3$ are in $\QQ$, then all eigenvalues of $A$ are rational, and so are the eigenvectors $u_1,u_2,u_3$. \qedhere
\end{proof}

\section{Non-finitely generated irrational simple cones}
\label{sec:4dim-example}

While in~ 2- and 3-dimensional cases (\Cref{sec:2dim,sec:3dim}), the necessary condition (\Cref{thm:polyhedral-necessary}) is also sufficient for a simple cone to be $(R,G)$-finitely generated, in this section we show that this is no longer true in the 4-dimensional case.

Let $C=\begin{pmatrix}3&2\\4&3 \end{pmatrix}$. The eigenvectors of $C$ are $e_1=\begin{pmatrix} 1\\ \sqrt{2}\end{pmatrix}$ and $e_2=\begin{pmatrix} -1\\ \sqrt{2}\end{pmatrix}$. Let $u_1=\begin{pmatrix} e_1\\0\end{pmatrix}$, $u_2=\begin{pmatrix} 0\\e_1\end{pmatrix}$, $u_3=\begin{pmatrix} e_2\\2e_2\end{pmatrix}$, $u_4=\begin{pmatrix} -e_2\\-e_2\end{pmatrix}$. Observe that $u_i$ are eigenvectors of the $4\times 4$ block-diagonal matrix $A=\begin{pmatrix} C&0\\0&C\end{pmatrix}$, where $u_1, u_2$ have the same eigenvalues $3+2\sqrt{2}$ and $u_3, u_4$ have the same eigenvalues $3-2\sqrt{2}$.

Let $M=\begin{pmatrix} 1& 0&-1&1\\ \sqrt{2}&0&\sqrt{2}&-\sqrt{2}\\0&1&-2&1\\0&\sqrt{2}&2\sqrt{2}&-\sqrt{2} \end{pmatrix}$
whose columns are $u_1,u_2,u_3,u_4$. Then coordinates $\alpha=(\alpha_1,\alpha_2,\alpha_3,\alpha_4)$ in the eigenvector basis translate into the vector $M\alpha$. Since we are interested in the case where $M\alpha \in \ZZ^4$ we may assume that $\alpha_i$ lie on the field $\QQ(\sqrt{2})$, i.e. we can write $\alpha_i=a_i+b_i\sqrt{2}$ where $a_i,b_i \in \QQ$. For $a+b\sqrt{2}\in \QQ(\sqrt{2})$ define the conjugate of $a+b\sqrt{2}$ as $a-b\sqrt{2}$.

Multiplying out $M\alpha$ we get

$$M\alpha=\begin{pmatrix} a_1-a_3+a_4+\sqrt{2}(b_1-b_3+b_4)\\2(b_1+b_3-b_4)+\sqrt{2}(a_1+a_3-a_4)\\ a_2 - 2 a_3 + a_4 + \sqrt{2}( b_2 - 2b_3 + b_4)\\   2(b_2 + 2b_3 -  b_4)+\sqrt{2}( a_2 + 2 a_3 -  a_4)\end{pmatrix}.$$

We can set the irrational parts of $M\alpha$ to 0 and solve for $a_3,b_3,a_4,b_4$ in terms of $a_1,b_1,a_2,b_2$ to get 
\begin{align}\label{eq:rel}
\begin{split}
a_3&=a_1-a_2\\ b_3&=b_2-b_1\\a_4&=2a_1-a_2\\b_4&=b_2-2b_1.
\end{split}
\end{align}
Substituting this into $M\alpha$ we get $$M\alpha=\begin{pmatrix}2a_1\\4b_1\\2a_2\\4b_2\end{pmatrix}.$$
From this we see that any point that satisfies \eqref{eq:rel} where $a_1,a_2$ are half-integral and $b_1,b_2$ are quarter-integral results in an integer point $M\alpha$. For simplicity of the notations, we can assume that 
$\alpha_1 = \frac{a_1}{2} + \frac{b_1\sqrt{2}}{4}$, $\alpha_2 = \frac{a_2}{2} + \frac{b_2\sqrt{2}}{4}$,
$\alpha_3 = \bar{\alpha}_1 - \bar{\alpha}_2$, $\alpha_4=2\bar{\alpha}_1 - \bar{\alpha}_2$, where $a_1,b_1,a_2,b_2\in\ZZ$.
This is an expression for the lattice $\alpha\in M^{-1}\ZZ^4$.

Before we show the example of $cone(u_1,u_2,u_3,u_4)$, we first show that a subset of this cone is not $(R,G)$-finitely generated.

\begin{proposition}\label{ex:subcone_4dim}
    The cone generated by $\{u_1, u_1+u_2, u_3, u_3+u_4\}$ is not $(R,G)$-finitely generated.
\end{proposition}
\begin{proof}
Suppose, for the sake of contradiction, that the cone is $(R,G)$-finitely generated.
By~\Cref{thm:polyhedral-necessary}, there is a matrix $B\in\GL(4,\ZZ)$ with the given eigenvectors $\{u_1, u_1+u_2, u_3, u_3+u_4\}$ and positive eigenvalues. 
Further, by~\Cref{lemma:NormalSubgroup}, it suffices to consider the group $G$ of such matrices.

First, we show that such matrix with the given eigenvectors $\{u_1, u_1+u_2, u_3, u_3+u_4\}$ must have two eigenvalues of multiplicity 2. More specifically, such matrix must be in the group generated by the matrix $A$.
This is because such matrix $B$ satisfies $B M' = M' \Lambda$, where $M'=[u_1, u_1+u_2, u_3, u_3+u_4]$, $\Lambda=\mathrm{diag}(\lambda_1,\lambda_2,\lambda_3,\lambda_4)$. As we are interested in the case where $B\in \ZZ^{4\times 4}$ and $M'\in\QQ(\sqrt{2})^{4\times4}$, we may assume assume that $\lambda_i$ lie on the field $\QQ(\sqrt{2})$, i.e. we can write $\lambda_i=x_i+y_i\sqrt{2}$ where $x_i,y_i \in \QQ$. Then 
\[\begin{aligned}
B = M'\Lambda M'^{-1}=
&\lambda_1
\begin{pmatrix}
\begin{array}{c|c}
\begin{matrix} 
\frac{1}{2} & \frac{\sqrt{2}}{4} \\[3pt] 
\frac{\sqrt{2}}{2} & \frac{1}{2} 
\end{matrix} 
& 
\begin{matrix} 
-\frac{1}{2} & -\frac{\sqrt{2}}{4} \\[3pt] 
-\frac{\sqrt{2}}{2} & -\frac{1}{2} 
\end{matrix} 
\\[3pt] \hline
\begin{matrix} 0 & 0 \\[3pt] 0 & 0 \end{matrix} 
& 
\begin{matrix} 0 & 0 \\[3pt] 0 & 0 \end{matrix}
\end{array}
\end{pmatrix}
+ \lambda_2
\begin{pmatrix}
\begin{array}{c|c}
\begin{matrix} 
0 & 0 \\[3pt] 0 & 0 
\end{matrix} 
& 
\begin{matrix} 
\frac{1}{2} & \frac{\sqrt{2}}{4} \\[3pt] 
\frac{\sqrt{2}}{2} & \frac{1}{2} 
\end{matrix} 
\\[3pt] \hline
\begin{matrix} 0 & 0 \\[3pt] 0 & 0 \end{matrix} 
& 
\begin{matrix} \frac{1}{2} & \frac{\sqrt{2}}{4} \\[3pt] \frac{\sqrt{2}}{2} & \frac{1}{2} \end{matrix}
\end{array}
\end{pmatrix}\\
& + \lambda_3
\begin{pmatrix}
\begin{array}{c|c}
\begin{matrix} 
\frac{1}{2} & -\frac{\sqrt{2}}{4} \\[3pt] 
-\frac{\sqrt{2}}{2} & \frac{1}{2} 
\end{matrix} 
& 
\begin{matrix} 
0 & 0 \\[3pt] 0 & 0 
\end{matrix} 
\\[3pt] \hline
\begin{matrix} 1 & -\frac{\sqrt{2}}{2} \\[3pt] -\sqrt{2} & 1 \end{matrix} 
& 
\begin{matrix} 0 & 0 \\[3pt] 0 & 0 \end{matrix}
\end{array}
\end{pmatrix}
+ \lambda_4
\begin{pmatrix}
\begin{array}{c|c}
\begin{matrix} 
0 & 0 \\[3pt] 0 & 0 
\end{matrix} 
& 
\begin{matrix} 
0 & 0 \\[3pt] 0 & 0 
\end{matrix} 
\\[3pt] \hline
\begin{matrix} -1 & \frac{\sqrt{2}}{2} \\[3pt] \sqrt{2} & -1 \end{matrix} 
& 
\begin{matrix} \frac{1}{2} & -\frac{\sqrt{2}}{4} \\[3pt] -\frac{\sqrt{2}}{2} & \frac{1}{2} \end{matrix}
\end{array}
\end{pmatrix}.
\end{aligned}\]
Since $B\in\mathbb{Z}^{4\times 4}$, we can get $\lambda_1 = \lambda_2 = x+y\sqrt{2}$, and $\lambda_3=\lambda_4=x-y\sqrt{2}$, where $x,y\in\mathbb{Z}$, and $B = \begin{pmatrix}
    x & y & 0 & 0\\
    2y & x & 0 & 0\\
    0 & 0 & x & y\\
    0 & 0 & 2y & x
\end{pmatrix}.$
Because all the eigenvalues are positive, we have $x \pm y\sqrt{2}\ge 0$. And $B$ is unimodular implying $x^2 - 2y^2 =1$. Such unimodular matrices are generated by the unimodular matrix $A$ when $x=3, y=2$ (see, for example, \cite{barbeau2003pell}).
We thus assume without loss of generality that the group $G=\langle A\rangle$.

Because any matrix from the group action has eigenvalues $\lambda_1=\lambda_2$ and $\lambda_3=\lambda_4$, we know that for any $M'\alpha\in\ZZ^4$ with $\alpha>0$, the ratio $\frac{\alpha_1}{\alpha_2}$ and $\frac{\alpha_3}{\alpha_4}$ will not change after any group action in $G$. 
Then, we show that there is no integer point $M'\alpha$ ($\alpha\ne 0$) in the cone such that $\alpha_2=0$, $\alpha_1,\alpha_3,\alpha_4\ge 0$. Otherwise, we may write one such point as $M'\alpha = M \begin{pmatrix}
    \alpha_1\\
    0\\
    \alpha_3+\alpha_4\\
    \alpha_4
\end{pmatrix}$, where $\alpha_1 = \frac{a_1}{2} + \frac{b_1\sqrt{2}}{4}$, $\alpha_3+\alpha_4 = \bar{\alpha}_1$, $\alpha_4=2\bar{\alpha}_1$, $a_1,b_1\in\ZZ$. Because $\alpha_3\ge0, \alpha_4\ge 0$, we know that $\bar{\alpha}_1=0$. This implies that $\alpha = 0$, a contradiction.

Suppose that $M'\alpha^*$ is the point in $R$ maximizing $\frac{\alpha^*_1}{\alpha^*_2}$, which exists because no integer points in the cone satisfy $\alpha_2= 0$ and $R$ is finite. Then we know that all the possible points generated by $R$ and the group action $G$ must satisfy $0\le \frac{\alpha_1}{\alpha_2}\le \frac{\alpha^*_1}{\alpha^*_2}$. However, because the cone generated by $u_1, (1+\epsilon)\alpha_1^* u_1 + \alpha^*_2 (u_1 + u_2), u_3, u_3 + u_4$ ($\epsilon>0$) is full dimensional, there must exists an integer point in the cone whose ratio $\frac{\alpha_1}{\alpha_2}\ge (1+\epsilon) \frac{\alpha^*_1}{\alpha^*_2}$.
This shows the contradiction to $(R,G)$-finite generation of the cone.
\end{proof}

Note that there are also no nonzero integer points in this cone satisfying $\alpha_1=0$ or $\alpha_3 = 0$ or $\alpha_4=0$. Such integer points can be viewed as an approximation to the extreme rays because such integer points can be balanced to be as close as possible to the extreme ray under the group action.


However, the cone can fail to be $(R,G)$-finitely generated even with the existence of integer vectors on the boundary approximating the extreme rays. The following example in \Cref{ex:cone_4dim} shows more obstructions for the $(R,G)$-finitely generation.

\medskip
\begin{proposition}\label{ex:cone_4dim}
    The cone generated by $\{u_1, u_2, u_3, u_4\}$ is not $(R,G)$-finitely generated.
\end{proposition}
\begin{proof}
Similarly to the argument in Example \ref{ex:subcone_4dim}, we can assume that the group $G$ is the group generated by the same matrix $A$.

Because any matrix from the group action has eigenvalues $\lambda_1=\lambda_2$ and $\lambda_3=\lambda_4$, we know that for any $M'\alpha\in\ZZ^4$ with $\alpha>0$, the ratio $\frac{\alpha_1}{\alpha_2}$ and $\frac{\alpha_3}{\alpha_4}$ will not change after any group action in $G$. 

However, unlike Example \ref{ex:subcone_4dim}, there exist integer vectors on the boundary of the cone approximating the extreme rays 
\[
\begin{array}{ll}
r_1 = \frac{1}{2} u_1 + \frac{1}{2} u_3 + u_4 = \begin{pmatrix} 1 \\ 0 \\ 0 \\ 0 \end{pmatrix}, \quad 
&r_2 = \frac{\sqrt{2}-1}{2} u_2 + \frac{\sqrt{2}+1}{2} u_3 + \frac{\sqrt{2}+1}{2} u_4 = \begin{pmatrix} 0 \\ 0 \\ -1 \\ 2 \end{pmatrix}, \\[8pt]
r_3 = \frac{\sqrt{2}-1}{2} u_1 + (\sqrt{2}-1) u_2 + \frac{\sqrt{2}+1}{2} u_3 = \begin{pmatrix} -1 \\ 2 \\ -2 \\ 4 \end{pmatrix}, \quad 
&r_4 = \frac{1}{2} u_1 + \frac{1}{2} u_2 + \frac{1}{2} u_4 = \begin{pmatrix} 1 \\ 0 \\ 1 \\ 0 \end{pmatrix}.
\end{array}.
\]

In general, for an integer point $M\beta\in\ZZ^4$ with $\beta=(\beta_1,\beta_2,\beta_3,\beta_4)$ from the boundary, we have
\begin{itemize}
    \item if $\beta_1 = 0$, then $\beta_2=-\bar{s}$, $\beta_3=s$, $\beta_4=s$, where $s\in\frac{a}{2} + \frac{b}{4}\sqrt{2}$, $a,b\in\ZZ$;
    \item if $\beta_2 = 0$, then $\beta_1=\bar{s}$, $\beta_3=s$, $\beta_4 = 2s$, where $s\in\frac{a}{2} + \frac{b}{4}\sqrt{2}$, $a,b\in\ZZ$;
    \item if $\beta_3 = 0$, then $\beta_1=s$, $\beta_2=s$, $\beta_4=\bar{s}$, where $s\in\frac{a}{2} + \frac{b}{4}\sqrt{2}$, $a,b\in\ZZ$;
    \item if $\beta_4 = 0$, then $\beta_1=s$, $\beta_2=2s$, $\beta_3 = -\bar{s}$, where $s\in\frac{a}{2} + \frac{b}{4}\sqrt{2}$, $a,b\in\ZZ$;
\end{itemize}

Because this cone is $(R,G)$-finitely generated, for any integer point $M\alpha\in\ZZ^4$ with $\alpha\ge 0$, there exists a point $M\beta\in G \cdot R$ with $0\le \beta \le\alpha$.
Because $R$ is finite, for any point $M\beta$ in $R$, either $\beta_2 = 0$, or $0\le \frac{\beta_1}{\beta_2}\le \frac{\alpha^*_1}{\alpha^*_2}$, where $M\alpha^*$ is the point in $R$ with $\beta_2\ne 0$ maximizing the ratio between the first two entries. Note that the ratio $\frac{\beta_1}{\beta_2}$ is preserved under the group actions in $G$, thus for any point $M\beta$ in $G\cdot R$, either $\beta_2 = 0$, or $0\le \frac{\beta_1}{\beta_2}\le \frac{\alpha^*_1}{\alpha^*_2}$. This implies that if $\frac{\alpha_1}{\alpha_2}>\frac{\alpha^*_1}{\alpha^*_2}$, there exists a point $M\beta\in G\cdot R$ with $\beta_2=0$ and $0\le \beta\le \alpha$.

To show that this cone is not $(R,G)$-finitely generated, we are going to construct an infinite family of points such that $\frac{\alpha_1}{\alpha_2}$ tends to infinity and there does not exist a nonzero vector from the boundary $\beta_2=0$ of the cone such that the point will remain within the cone after subtracting the boundary vector. In this way, there exists a point in this family with $\frac{\alpha_1}{\alpha_2}>\frac{\alpha^*_1}{\alpha^*_2}$. Therefore, there must exist a point $M\beta\in G\cdot R$ with $\beta_2=0$ and $0\le \beta\le \alpha$. This is a contradiction to the construction of these points.

Pick $v^{(n_1)}=M\alpha$, where $\alpha_2=\frac{1}{2}(\sqrt{2}-1)^{2n_1+1}$, $\alpha_4 = \frac{1}{2}$, where $n_1\in\mathbb{Z}_{\ge 0}$, and
\begin{align*}
    &\alpha_1=\frac12\left(\alpha_2 + \bar{\alpha}_4\right)=\frac12\left(\frac{1}{2}(\sqrt{2}-1)^{2n_1+1}+\frac{1}{2}\right),\\
    &\alpha_3=\frac12\left(-\bar{\alpha}_2 + \alpha_4\right)=\frac12\left(\frac{1}{2}(\sqrt{2}+1)^{2n_1+1}+\frac{1}{2}\right).
\end{align*}
Then $\alpha_3 = \bar{\alpha}_1 -\bar{\alpha}_2$, $\alpha_4 = 2\bar{\alpha}_1-\bar{\alpha}_2$.
Assume that $\alpha_i = \frac{a_i}{2} + \frac{b_i}{4}\sqrt{2}$, $i=1,2,3,4$.
By the binomial theorem, $(\sqrt{2}-1)^{2n_1 + 1} = x_i + y_i\sqrt{2}$ for $x_i, y_i\in\ZZ$, which implies that $a_2\in\ZZ$, $\frac{b_2}{2}\in\ZZ$. Also, $\bar{\alpha}_2 = -\frac12(\sqrt{2}+1)^{2n_1+1}=\frac{a_2}{2} - \frac{b_2}{4}\sqrt{2}$, which implies that $a_2^2 - 2(b_2/2)^2=-4\alpha_2\bar{\alpha}_2=-1$. Thus, $a_2$ is odd and $b_2$ is even. Therefore, $a_1 = \frac12(a_2 + 1)\in\ZZ$, and $b_1 = \frac12 b_2\in\ZZ$. These verifications show that the constructed $v^{(n_1)}$ is an integer point in the cone, where
$$
v^{(n_1)} = M\alpha = \begin{pmatrix}\frac{a_2 + a_4}{2}\\\frac{b_2-b_4}{2}\\a_2\\b_2\end{pmatrix} \in\ZZ^4.
$$
We claim that we cannot subtract from $v^{(n_1)}$ any nonzero vector from the boundary $\beta_2=0$ of the cone, and remain within the cone.

From the construction of $v^{(n_1)}$, we have $0<\alpha_2\le \alpha_1$, $\alpha_2\bar{\alpha}_2 = -\frac14$. And $\alpha_1\alpha_4 = \frac14\alpha_2 + \frac18 < \frac14$.
Also $\frac{\alpha_1}{\alpha_2} = \frac12 +\frac12(\sqrt{2}-1)^{2n_1 + 1}$, which tends to infinity as $n_1$ increases. 

Suppose that we can subtract a boundary vector $M\beta$ with $\beta=(\bar{s}, 0, s, 2s)$, where $s = \frac{a}{2}+\frac{b}{4}\sqrt{2}$, $a,b\in\ZZ$.
Then we must have $0\le \bar{s}\le \alpha_1$, $0\le 2s\le \alpha_4$. By multiplying these two inequalities, we have $0\le 2\bar{s}s=\frac{2a^2-b^2}{4}\le \alpha_1\alpha_4<\frac14$. This forces $2a^2-b^2=0$, which only has the trivial solution $0$.

Therefore, we cannot subtract from $v^{(n_1)}$ any nonzero vector on the boundary $\beta_2=0$, and remain within the cone. The construction of this family of points will complete the proof.
\end{proof}

\section{Acknowledgements:} We are grateful for suggestions we received from Brittney Marsters and Bjorn Poonen. The second and third authors are grateful for financial support received from NSF grants DMS-2348578 and DMS-2434665. The last three authors are grateful for the support received by NSF Grant DMS-1929284 of ICERM.

\printbibliography
\end{document}